\documentclass[journal,twoside,web]{ieeecolor}
\usepackage{lcsys}
\usepackage{xcolor}
\usepackage{cite}
\usepackage{amsmath,amssymb,amsfonts}

\usepackage{amsthm}
\usepackage{graphicx}
\usepackage{textcomp}
\usepackage[utf8]{inputenc}
\usepackage{float}
\usepackage{subcaption}
\usepackage[ruled, vlined, linesnumbered]{algorithm2e}
\usepackage{booktabs}
\DontPrintSemicolon
\usepackage[left=2cm,right=2cm, top=2cm, bottom=2cm]{geometry}
\usepackage{url}
\usepackage{courier}
\usepackage{pgf}

\newtheorem{lemma}{Lemma}
\newtheorem{theorem}{Theorem}

\DeclareMathSymbol{\shortminus}{\mathbin}{AMSa}{"39}

\usepackage{booktabs}
\usepackage{optidef}
\usepackage{comment}

\def\BibTeX{{\rm B\kern-.05em{\sc i\kern-.025em b}\kern-.08em
    T\kern-.1667em\lower.7ex\hbox{E}\kern-.125emX}}
\begin{document}

\title{Fast Generation of Feasible Trajectories in Direct Optimal Control}
\author{David Kiessling$^{1}$, Katrin Baumg\"artner$^{2}$, Jonathan Frey$^{2}$, Wilm Decr\'e$^{1}$, Jan Swevers$^{1}$, Moritz Diehl$^{2, 3}$
\thanks{$^{1}$Dept.~of~Mechanical~Engineering,~KU~Leuven and Flanders Make@KU Leuven,
        {\tt\small david.kiessling@kuleuven.be}}%
\thanks{$^{2}$Dept. of Microsystems Engineering, University of
Freiburg}%
\thanks{$^{3}$Dept. of Mathematics, University of Freiburg.}%
\thanks{This work has been carried out within the framework of the Flanders
Make SBO project DIRAC: Deterministic and Inexpensive Realizations of Advanced
Control.}
}

\maketitle
\thispagestyle{empty}
\pagestyle{empty}
\begin{abstract}
    This paper examines the question of finding feasible points to discrete-time optimal control problems. 
    The optimization problem of finding a feasible trajectory is transcribed to an unconstrained optimal control problem. 
    An efficient algorithm, called FP-DDP, is proposed that solves the resulting problem using Differential Dynamic Programming preserving feasibility with respect to the system dynamics in every iteration. 
    Notably, FP-DDP admits global and rapid local convergence properties induced by a combination of a Levenberg-Marquardt method and an Armijo-type line search. 
    An efficient implementation of FP-DDP within \texttt{acados} is compared to established methods such as Direct Multiple Shooting, Direct Single Shooting, and state-of-the-art solvers.
\end{abstract}
\section{Introduction}
\label{sec:introduction} 
\IEEEPARstart{T}{his} paper analyzes the problem of finding feasible points of Nonlinear Optimization Problems (NLPs) arising in discrete-time optimal control, 
which are given by:
\begingroup\abovedisplayskip=0pt \belowdisplayskip=3pt
\begin{mini!}|s|
{\scriptstyle{\substack{x_1, \ldots, x_{N+1}, \\ u_1, \ldots, u_N}}}
{\sum_{i=1}^{N} l_i\left(x_i, u_i\right)+ l_{N+1}\left(x_{N+1}\right)}
{\label{eq:nlp_ocp}}
{}
\addConstraint{}{\!\!\!\!x_1= \bar{x}_1, x_{i+1}=\phi_i\left(x_i, u_i\right),}{i=1,\ldots, N\! \label{eq:dynamics_constraint}}
\addConstraint{}{\!\!\!\!c_i\left(x_i, u_i\right) \leq 0,}{i=1,\ldots,N\!}
\addConstraint{}{\!\!\!\!c_{N+1}\left(x_{N+1}\right) \leq 0.}
\end{mini!}
\endgroup
In the above definition, the state vectors are denoted by $x_i \in \mathbb{R}^{n_x}$, the controls by $u_i\in \mathbb{R}^{n_u}$, and the initial state by $\bar{x}_1\in\mathbb{R}^{n_x}$.
The stage and terminal cost are $l_i$ and $l_{N+1}$, respectively.
The function $\phi_i$ is the discrete-time dynamics which yields the state $x_{i+1}$ at the next stage given the current state and control $x_i, u_i$.
The remaining constraints are the path and terminal constraints $c_i$ and $c_{N+1}$.

Reducing infeasibility and finding feasible points is a key building block of constrained optimization algorithms. 
State-of-the-art nonlinear optimization solvers such as \texttt{IPOPT}~\cite{Waechter2006} or \texttt{filterSQP}~\cite{fletcher1997} employ a feasibility restoration phase to avoid cycling, to overcome infeasible subproblems, and failure of sufficient decrease by iterating towards the feasible set.
Feasible iterate solvers, such as \texttt{SEQUOIA}~\cite{Nita2023} or \texttt{FSQP}~\cite{Tits1992}, require every iterate, including the initial guess, to be feasible.
An initial feasibility phase which solves a feasibility problem is often employed to find a feasible initial guess to a generic NLP.

The feasibility-projected Sequential Quadratic Programming (SQP) framework \cite{Tenny2004}, \cite{Wright2004} is a feasible iterate method that exploits structure to solve linearly constrained Optimal Control Problems (OCPs). 
This method employs a trust-region SQP method that uses a feasibility projection to achieve feasibility in every iteration.
The feasibility projection is based on Differential Dynamic Programming (DDP) to guarantee feasibility with respect to the dynamics constraints while solving an additional optimization problem to maintain feasibility with respect to linear inequality constraints.
Similarly, the method proposed in this paper is based on DDP but relies on reformulating additional inequality constraints as quadratic penalties in the objective function.

DDP, originally introduced in \cite{mayne1966}, is a method for solving unconstrained OCPs that performs a Riccati backward recursion for a quadratic approximation of the unconstrained OCP and a nonlinear forward simulation to achieve feasibility with respect to the dynamics in every iteration. 
Local quadratic convergence for exact Hessian DDP was first shown in \cite{murray1984}.
The first global convergence proof for line search DDP was presented in \cite{Yakowitz1984}. In \cite{Baumgaertner2023a}, it was shown that DDP, Direct Multiple Shooting (DMS), and Direct Single Shooting (DSS) recover the same local convergence rate when used with a Generalized Gauss-Newton (GGN) Hessian approximation.
Constrained optimization algorithms require a globalization mechanism, where both the decrease of the objective function and the infeasibility of the constraints are controlled. 
Since the relative weighting of these two terms is nontrivial, DDP is used in this paper as a means of avoiding globalization of the dynamics constraints.

The contributions of this paper are threefold. 
Firstly, the optimization problem of finding a feasible trajectory is discussed starting from state-of-the-art solvers and a suitable structure-preserving feasibility problem formulation is identified for the special case of OCPs.
Secondly, based on the proposed optimization problem, a globally convergent DDP algorithm is introduced. 
Finally, the novel feasibility algorithm is compared against state-of-the-art methods, DMS, and DSS, on challenging OCPs.

The remainder of this paper is structured as follows.
Section~\ref{sec:problem_statement_and_discussion} discusses the optimization problem of finding a feasible point.
In Section~\ref{sec:algorithm_section}, the novel feasibility algorithm is introduced.
Simulation results are presented in Section~\ref{sec:simulation_results}.
Section \ref{sec:conclusion} concludes this paper. 
\paragraph*{Notation}
The concatenation of vectors $w_1\in\mathbb{R}^{n_1}$ to $w_n\in\mathbb{R}^{n_n}$ is denoted by $(w_1,\ldots,w_n)\in\mathbb{R}^{n_1+\ldots+n_n}$. 
Iteration indices are denoted by a letter $k$ and superscripts in parentheses, e.g., $w^{(k)}\in\mathbb{R}^{n_w}$.
For a given iterate $w^{(k)}$, quantities evaluated at that iterate will be denoted with the same superscript, e.g., $f^{(k)}$ or $\nabla f^{(k)}$.
In the context of optimal control problems, an OCP stage is denoted by a letter $i$ and subscript, e.g., $x_i$ and $u_i$.  Let $[\cdot]^+ \colon \mathbb{R}^n \to \mathbb{R}^n, x\mapsto \max\{0, x\}$ denote the componentwise maximum of a vector with the zero vector.
For functions $f,g\colon\mathbb{R}^{n_w}\to\mathbb{R}$, $f(w)=\mathcal{O}(g(w))$ means that there exists a constant $C>0$ such that $\vert f(w)\vert \leq C g(w)$ in a neighborhood around $0$.
\section{Problem Statement}
\label{sec:problem_statement_and_discussion}
Before introducing a favorable problem formulation which will be the basis of the proposed algorithm, a brief overview of related approaches is provided.
To this end, the vector of constraint violations for \eqref{eq:nlp_ocp} is given by $v(w) := (v_0, \ldots, v_{N+1})$ where $w=(x_1, \dots, x_{N+1}, u_1, \dots, u_N)$
and the stagewise constraint violations are
$v_0 := x_1 - \bar x_1$,
$v_i := (x_{i+1} - \phi_i(x_i, u_i), [c_i(x_i, u_i)]^+)$, $i=1, \ldots, N$, and
$v_{N+1} = [c_{N+1}(x_{N+1})]^+$.
A point $w$ is feasible for \eqref{eq:nlp_ocp} if and only if $\Vert v(w) \Vert=0$ for a given norm $\Vert\cdot\Vert$. One way of formulating the problem of finding a feasible point for \eqref{eq:nlp_ocp} is thus
\begingroup\abovedisplayskip=3pt \belowdisplayskip=3pt
\vspace{-3mm}
\begin{align}
\label{eq:general_feasibility_problem}
    \min_{w}\quad \Vert v(w)\Vert.
\end{align}
\endgroup
\subsection{Related Literature}
Several successful nonlinear programming solvers solve the feasibility problem \eqref{eq:general_feasibility_problem}. 
This may be during a feasibility restoration phase as for \texttt{filterSQP} and \texttt{IPOPT}.
The feasible iterate solvers \texttt{FSQP} and \texttt{SEQUOIA} solve the feasibility problem before starting their main algorithm from a feasible point. 
An overview of the feasibility problems used in the different solvers is given in Table \ref{tab:comparison_feasibility_problems}.
\begingroup\abovedisplayskip=0.5pt \belowdisplayskip=0.5pt
\begin{table}[]
\caption{Comparison of feasibility problems for different state-of-the-art solvers.}
\label{tab:comparison_feasibility_problems}
\begin{tabular}{l l l}
\toprule
Solver & Problem & Method \\ 
\midrule
\texttt{IPOPT} & $\min_w \Vert v(w)\Vert_1 + \frac{\rho}{2}\Vert w-w^{(0)}\Vert_2^2$ & Interior Point\\[3pt]
\texttt{filterSQP} & $\min_w \Vert v(w)\Vert_{1}$ & SQP\\[3pt]
\texttt{FSQP} & $\min_w \Vert v(w)\Vert_{\infty}$ & SQP\\[3pt]
\texttt{SEQUOIA} & $\min_w \frac{1}{2}\Vert v(w)\Vert_2^2$ & Unconstrained\\
\bottomrule
\end{tabular}
\vspace{-6mm}
\end{table}
\endgroup

Firstly, consider the feasibility restoration problem of \texttt{IPOPT} as given in Table~\ref{tab:comparison_feasibility_problems}.
The feasibility restoration problem contains in addition to the minimization of the constraint violation, a regularization term $\frac{\rho}{2}\Vert w-w^{(0)}\Vert_2^2$ with $\rho>0$.
If $\rho$ is chosen small enough, the \texttt{IPOPT} subproblem is the exact penalty formulation of finding the closest feasible point to $w^{(0)}$ \cite{Waechter2006}. 
Note that the correct choice of $\rho$ is not known beforehand. 
The first term $\Vert v(w)\Vert_1$ is nonsmooth and is handled in \texttt{IPOPT} or \texttt{filterSQP} by a smooth reformulation introducing slack variables in every constraint. 
Another key feature of introducing slack variables is that the Linear Independence Constraint Qualification (LICQ) is fulfilled at every feasible point and that the underlying subproblems are always feasible. 
The Hessian of the Lagrangian for the feasibility problem is in general indefinite and it is necessary to properly deal with its negative curvature. 
An important observation for OCPs is that the gradients of the constraints \eqref{eq:dynamics_constraint} are always linearly independent and build a constraint Jacobian with full row rank in an SQP subproblem. 
Therefore, it is not strictly necessary to introduce slack variables for these constraints.

Secondly, the \texttt{SEQUOIA} solver minimizes $\frac{1}{2}\Vert v(w)\Vert_2^2$ in its first iteration to find a feasible point.
A favorable property of this formulation is that it has a nonlinear least-squares objective. 
Due to this structure, no additional reformulation is required to solve the unconstrained optimization problem and a GGN Hessian approximation can be used. 
The outer convexity of the GGN Hessian is $\tfrac{1}{2}\Vert[\cdot]^+\Vert_2^2$ and the inner nonlinearities are the constraint functions $c_{i}, \phi_i$.
Assuming that the constraints of \eqref{eq:nlp_ocp} are continuously differentiable, then the \texttt{SEQUOIA} function is only once continuously differentiable and standard analysis needs to be adapted resulting in the framework of semismooth Newton methods, cf. e.g.,~\cite{Hintermueller2010}.
\subsection{A Favorable Problem Formulation}
Combining the advantages of the previously discussed feasibility problems and considering the OCP structure, an appealing structure-preserving formulation for a feasibility problem is the following Feasibility Optimal Control Problem (FOCP):
\begingroup\abovedisplayskip=2pt \belowdisplayskip=2pt
\begin{mini!}|s|
{\scriptstyle{\mathbf{x}, \mathbf{u}}}
{f(\mathbf{x},\mathbf{u}):=\textstyle\sum_{i=1}^{N} f_i\left(x_i, u_i\right)+ f_{N+1}\left(x_{N+1}\right)}{\label{eq:nlp_focp}}
{}
\addConstraint{x_{1}}{=u_0,\:\:x_{i+1}=\phi_i\left(x_i, u_i\right),~\label{eq:focp_dynamics_constraint}}{i=1,\ldots,N,}
\end{mini!}
\endgroup
where 
\begingroup\abovedisplayskip=2pt \belowdisplayskip=2pt
\begin{align*}
    &f_1(x_1,u_1) :=\tfrac{1}{2}\left(\Vert x_1 - \bar{x}_1\Vert_2^2 + \Vert [c_1\left(x_1, u_1\right)]^+\Vert_2^2\right),\\
    &f_i(x_i,u_i) := \tfrac{1}{2}\Vert [c_i\left(x_i, u_i\right)]^+\Vert_2^2,~i=2,\ldots, N,\\
    &f_{N+1}(x_{N+1}) := \tfrac{1}{2}\Vert [c_{N+1}\left(x_{N+1}\right)]^+\Vert_2^2,
\end{align*}
\endgroup
with $\mathbf{u}:=(u_0,\ldots,u_{N})$ and $\mathbf{x}:=(x_1,\ldots,x_{N+1})$.
Note that an auxiliary control $u_0\in\mathbb{R}^{n_x}$ was introduced and the initial condition was softened in the objective function $f_1$ such that \eqref{eq:nlp_focp} is an unconstrained OCP without initial condition. 
This choice is motivated by applications such as model predictive control where the solver is often initialized with the solution of the previous time step which is feasible for all constraints except for the initial condition. Another application are Point-to-Point (P2P) motions where the task is to bring a system from initial state $\bar{x}_1$ to terminal state $\bar{x}_{N+1}$. Since the initial state and the terminal state should be treated equally, the initial condition constraint is also softened in the objective function.
In general, the optimal solution of \eqref{eq:nlp_focp} is not unique.
All feasible points of \eqref{eq:nlp_ocp} are an optimal solution to \eqref{eq:nlp_focp} and fulfill $f(\mathbf{x}, \mathbf{u})=0$.
If the states are eliminated using~\eqref{eq:focp_dynamics_constraint}, the objective function only expressed by the controls is denoted by $F(\mathbf{u})$.

Formulation \eqref{eq:nlp_focp} has the following favorable properties:
(i) keeping the dynamics constraints as constraints preserves the same problem dimension as those of the original OCP and ensures LICQ satisfaction at every feasible point of \eqref{eq:nlp_focp}; 
(ii) given the unconstrained OCP structure of \eqref{eq:nlp_focp}, a Riccati recursion can be used to efficiently solve the subproblems;
(iii) due to the nonlinear least-squares objective function in~\eqref{eq:nlp_focp}, a GGN Hessian approximation can be used to approximate the Hessian of the Lagrangian of~\eqref{eq:nlp_focp};
(iv) since the constraint violation is zero at any optimal solution, the exact Hessian and the GGN Hessian coincide. 
This implies that local superlinear convergence towards the optimal solution can be achieved.
\section{FP-DDP Algorithm}
\label{sec:algorithm_section}
In this section, the algorithm for finding feasible trajectories, FP-DDP, is described.
The proposed method combines DDP with a tailored backtracking line search to guarantee global convergence.
FP-DDP is summarized in Algorithm~\ref{alg:ffp}.
The following sections discuss its building blocks.
\subsection{Differential Dynamic Programming}
\label{sec:ddp}
Given the current iterate $(\bar{\mathbf{x}}, \bar{\mathbf{u}})$ and referring to \cite{Baumgaertner2023a}, a simultaneous method linearizes \eqref{eq:nlp_focp} at $(\bar{\mathbf{x}}, \bar{\mathbf{u}})$ yielding:
\begingroup\abovedisplayskip=0.5pt \belowdisplayskip=1.5pt
\begin{mini!}|s|
{\scriptstyle{\mathbf{\delta x}, \mathbf{\delta u}}}
{\sum_{i=1}^{N}\begin{bmatrix}
q_i \\
r_i
\end{bmatrix}^{\top}\!\begin{bmatrix}
\delta x_i \\
\delta u_i
\end{bmatrix}+\frac{1}{2}\begin{bmatrix}
\delta x_i \\
\delta u_i
\end{bmatrix}^{\top}\!\begin{bmatrix}
Q_i & S_i^{\top} \\
S_i & R_i
\end{bmatrix}\begin{bmatrix}
\delta x_i \\
\delta u_i
\end{bmatrix}\nonumber}
{\label{eq:ocp_qp}}
{}
\breakObjective{+p_{N+1}^{\top} \delta x_{N+1}+\tfrac{1}{2} \delta x_{N+1}^{\top} P_{N+1} \delta x_{N+1}}
\addConstraint{\delta x_1}{= \delta u_0}
\addConstraint{\delta x_{i+1}}{=A_i \delta x_i+B_i \delta u_i +b_i,~i=1,\ldots,N,}
\end{mini!}
\endgroup
where $\delta x_i = x_i - \bar{x}_i$ for $i=1,\ldots,N+1$, $\delta u_i = u_i - \bar{u}_i$ for $i=0,\ldots,N$. The concatenation of all delta terms is denoted by $(\mathbf{\delta x}, \mathbf{\delta u})$. The dynamics terms are given by
\begingroup\abovedisplayskip=3pt \belowdisplayskip=3pt
\begin{align*}
A_i =\tfrac{\partial \phi_i}{\partial x_i}\left(\bar{x}_i, \bar{u}_i\right),~B_i=\tfrac{\partial \phi_i}{\partial u_i}\left(\bar{x}_i, \bar{u}_i\right),
\end{align*}
\endgroup
and $b_i =\phi_i\left(\bar{x}_i, \bar{u}_i\right) - \bar{x}_{i+1}$. The gradient terms in the objective function are defined by
$q_i =\nabla_{x_i} f_i(\bar{x}_i$, $\bar{u}_i), r_i =\nabla_{u_i} f_i(\bar{x}_i, \bar{u}_i)$, and $p_{N+1} =\nabla_{x_{N+1}} f_{N+1}(\bar{x}_{N+1})$.
The Hessian parts are defined by $P_{N+1}=H_{N+1}(\bar{x}_{N+1})$ and
\begingroup\abovedisplayskip=3pt \belowdisplayskip=3pt
\begin{align}
\label{eq:splitting_hessian}
    \begin{bmatrix}Q_i & S_i^{\top} \\ S_i & R_i\end{bmatrix}=H_i(\bar{x}_i, \bar{u}_i).
\end{align}
\endgroup
Defining
$R_0:= 0$, 
$r_0:= 0$, 
$B_0:= I$,
and
$D_{i}:= R_i + B_i^{\top}P_{i+1}B_i$,
$d_{i}:= r_i + B_i^{\top}(P_{i+1}b_i + p_{i+1})$, 
the QP \eqref{eq:ocp_qp} can be efficiently solved by a Riccati recursion algorithm. First, a backward sweep calculates the following matrices
\begingroup\abovedisplayskip=3pt \belowdisplayskip=3pt
\begin{align}
k_i & =-D_i^{-1}d_i, ~~ i=N, \ldots, 0, \label{eq:small_K_backward recursion}\\
K_i & =-D_i^{-1}\left(S_i+B_i^{\top} P_{i+1} A_i\right), \label{eq:big_K_backward recursion}\\
P_i &= Q_i+A_i^{\top} P_{i+1} A_i+\left(S_i^{\top}+A_i^{\top} P_{i+1} B_i\right) K_i,\\
p_i &= q_i+A_i^{\top}\left(P_{i+1} b_i+p_{i+1}\right)+K_i^{\top}d_i, \label{eq:small_P_backward_recursion}
\end{align}
\endgroup
for $i=N,\ldots,1$.
Then, a forward sweep calculates the solution of the above QP with
\begingroup\abovedisplayskip=3pt \belowdisplayskip=3pt
\begin{subequations}
\begin{align}
u_0 & =\bar{u}_0 + \alpha k_0, u_i =\bar{u}_i + \alpha k_i+K_i\left(x_i-\bar{x}_i\right), \label{eq:dms_control}\\
x_{i+1} & =A_i\left(x_i-\bar{x}_i\right)+B_i\left(u_i-\bar{u}_i\right) + b_i + \bar{x}_{i+1}, \label{eq:linearized_rollout}
\end{align}
\label{eq:all_dms_equations}
\end{subequations}
\endgroup
for $i=1, \ldots, N$ and $\alpha=1$. Instead of using the linearized dynamics \eqref{eq:linearized_rollout} in the forward rollout, DDP uses the nonlinear dynamics of \eqref{eq:nlp_focp} in the forward simulation.
This yields the nonlinear rollout
\begingroup\abovedisplayskip=3pt \belowdisplayskip=3pt
\begin{align}
x_1 = u_0, ~~x_{i+1} & =\phi_i(x_i, u_i). \label{eq:ddp_forward}
\end{align}
\endgroup
DDP requires $(\bar{\mathbf{x}}, \bar{\mathbf{u}})$ to be dynamically feasible. This includes the initial guess. Due to the nonlinear rollout, all subsequent iterates of DDP are also dynamically feasible.
 As shown in \cite{Mastalli2019}, the optimal objective $f_{\mathrm{QP}}^*(\bar{\mathbf{x}}, \bar{\mathbf{u}})$ of \eqref{eq:ocp_qp} can be efficiently calculated by
 \begingroup\abovedisplayskip=3pt \belowdisplayskip=4pt
\begin{align}
\label{eq:qp_optimal_sol}
    f_{\mathrm{QP}}^*(\bar{\mathbf{x}},\bar{\mathbf{u}})\! = \!\sum_{i=0}^N \tfrac{1}{2}k_i^{\top}D_i k_i  \!+\! k_i^{\top} \! d_i = \!-\!\!\sum_{i=0}^N \tfrac{1}{2} d_i^{\top}D_i^{\shortminus 1} d_i.
\end{align}
\endgroup
Finally, it is noted that the Riccati recursion is well-defined if $D_i$ is invertible. This can be guaranteed by choosing positive definite matrices $H_i$ \cite{Yakowitz1984}. The subsequent section illustrates how this is ensured by FP-DDP.
\subsection{Regularization}
\label{sec:regularization}
The GGN Hessian terms for $f_i,i=1\ldots,n+1$ are denoted by $H_i^{\mathrm{GGN}}$.
In \eqref{eq:splitting_hessian}, the GGN Hessian is combined with an adaptively reduced Levenberg-Marquardt regularization term taken from \cite{bergou2020}, i.e.,
\begingroup\abovedisplayskip=3pt \belowdisplayskip=3pt
\begin{align}
\label{eq:regularization}
    H^{(k)}_i = (H^{\mathrm{GGN}}_i)^{(k)} + \gamma^{(k)}I_i,
\end{align}
\endgroup
where $\gamma^{(k)}:=\mu^{(k)}f(\mathbf{x}^{(k)},\mathbf{u}^{(k)})$ for $\mu^{(k)}>0$ and $I_i$ denotes the identity matrix of appropriate size. 
If a full step is taken, the parameter $\mu$ is updated as $\mu^{(k+1)}\gets \max\{\mu_{\min}, \tfrac{\bar{\mu}}{\lambda}\}$ and $\bar{\mu}\gets\mu^{(k)}$, where $\mu_{\min}>0, \lambda>1$ are fixed parameters, $\mu^{(0)}>\mu_{\min}$, and $\bar{\mu}$ is initialized with $\mu^{(0)}$. If a smaller step is taken, then $\mu^{(k+1)}\gets \lambda\mu^{(k)}$. The regularization term vanishes as the iterates approach a feasible point, allowing for superlinear local convergence \cite{bergou2020}. Since $H^{\mathrm{GGN}}_i$ is always positive semi-definite and since the regularization term is positive as long as the algorithm has not converged, $H^{(k)}_i$ defined by \eqref{eq:regularization} is ensured to be positive definite.
If the step size $\alpha$ goes below a given minimum $\alpha_{\min}$ or if the backward Riccati sweep fails, the regularization parameter $\gamma$ is increased similarly to \cite{Mastalli2019}, by multiplying $\lambda$ with $\mu$ and restarting the DDP iteration with the updated regularization parameter.
Since a GGN Hessian approximation is used the proposed algorithm is free of Lagrange multipliers.
\subsection{Line Search and Step Acceptance}
\label{sec:our_algorithm}
Due to the use of DDP, globalization is only required for the objective function which is once differentiable due to the choice of the squared 2-norm.
A prominent feature of this choice is that the Maratos effect is avoided \cite{Nocedal2006}.
The backward Riccati recursion on problem \eqref{eq:ocp_qp} yields the gain matrices $K_i$ and the feedforward matrices $k_i$. The forward sweep, given by \eqref{eq:dms_control} and \eqref{eq:ddp_forward}, includes the line search parameter $\alpha\in(0,1]$.
An Armijo condition decides upon step acceptance. 
Therefore, the predicted reduction of the quadratic model of $f$ at iterate $(\mathbf{x}^{(k)}, \mathbf{u}^{(k)})$ using problem \eqref{eq:ocp_qp} and its solution $(\mathbf{\delta x}, \mathbf{\delta u})$ is defined by
\begin{equation}
    \resizebox{.87\hsize}{!}{$m_f^{(k)}\!(\mathbf{\delta x},\mathbf{\delta u})\!=\!\shortminus\!\left((\nabla f^{(k)})^{\top}\!\!+\!\frac{1}{2} (\mathbf{\delta x},\mathbf{\delta u})^{\top} H^{(k)}\right)(\mathbf{\delta x},\mathbf{\delta u})$}.
\end{equation}
Note that $m_f^{(k)}(\mathbf{\delta x}, \mathbf{\delta u})=-f^*_{\mathrm{QP}}(\mathbf{x}^{(k)}, \mathbf{u}^{(k)})$. Given the iterate $(\mathbf{x}^{(k)}, \mathbf{u}^{(k)})$, the trial iterate for a given $\alpha$ is denoted by $(\mathbf{\hat{x}}(\alpha), \mathbf{\hat{u}}(\alpha))$.
The Armijo condition is defined by 
\begingroup\abovedisplayskip=3pt \belowdisplayskip=3pt
\begin{align}
\label{eq:sufficient_decrease}
    f^{(k)} - f(\mathbf{\hat{x}}(\alpha), \mathbf{\hat{u}}(\alpha)) \geq \eta\alpha m_f^{(k)}(\mathbf{\delta x}, \mathbf{\delta u})
\end{align}
\endgroup
for $\eta\in(0,1)$.
If the current step size is not acceptable for \eqref{eq:sufficient_decrease}, the step size is reduced by a factor $\tfrac{1}{2}$, i.e., $\alpha\gets \tfrac{1}{2}\alpha$.
\subsection{Initialization and Termination}
As mentioned in Section \ref{sec:ddp}, DDP requires a dynamically feasible initialization. 
In case such a guess is not provided, prior to starting FP-DDP, a Riccati backward sweep \eqref{eq:small_K_backward recursion} - \eqref{eq:small_P_backward_recursion} is performed based on the infeasible initial guess and a nonlinear forward simulation \eqref{eq:dms_control} and \eqref{eq:ddp_forward} yields the dynamically feasible initialization.
The algorithm stops, if a feasible trajectory is found, i.e., the objective function is below a given tolerance $F(\mathbf{u})\leq\varepsilon_{\mathrm{F}}$ or if a stationary point was found, i.e., $\Vert\nabla F(\mathbf{u})\Vert_{\infty}\leq \varepsilon_{\mathrm{S}}$.
\begingroup\abovedisplayskip=3pt \belowdisplayskip=3pt
\begin{algorithm}[thpb]
\caption{Feasibility Projection (FP)-DDP}
\label{alg:ffp}
\SetKwInOut{Parameter}{Parameter}
\Parameter{$\eta,\alpha_{\min},\varepsilon_{\mathrm{F}},\varepsilon_{\mathrm{S}},\mu_{\min}\in(0,1), \lambda > 0;$}
\KwIn{Initial guess  $(\mathbf{\hat{x}}^{(0)}, \mathbf{\hat{u}}^{(0)})$, $\mu^{(0)}$}
Find dynamically feasible initial guess $(\mathbf{x}^{(0)}, \mathbf{u}^{(0)})$\;
\For{$k=0, 1, 2, \ldots$}{
\lIf{$F(\mathbf{u})\leq\varepsilon_{\mathrm{F}}$ or $\Vert\nabla F(\mathbf{u})\Vert_{\infty}\leq\varepsilon_{\mathrm{\mathrm{S}}}$}{\textbf{break}}
Setup QP \eqref{eq:ocp_qp} and calculate $K_i, k_i$ from \eqref{eq:big_K_backward recursion}, \eqref{eq:small_K_backward recursion}\;
$\alpha^{(k)}\gets 1.0$, calculate $m_f^{(k)}(\mathbf{\delta x}, \mathbf{\delta u})$ with \eqref{eq:qp_optimal_sol}\;
\While{$\mathrm{True}$}{
\lIf{$\alpha^{(k)} <\alpha_{\min}$}{
$\mu^{(k+1)}\gets\lambda\mu^{(k)}$;
\textbf{break}
}
Calculate $(\mathbf{\hat{x}}^{(k)}(\alpha), \mathbf{\hat{u}}^{(k)}(\alpha))$ from \eqref{eq:dms_control},  \eqref{eq:ddp_forward}\;
\uIf{Sufficient decrease \eqref{eq:sufficient_decrease} is satisfied}{
$(\mathbf{x}^{(k+1)}, \mathbf{u}^{(k+1)})\gets(\mathbf{\hat{x}}^{(k)}(\alpha), \mathbf{\hat{u}}^{(k)}(\alpha))$\;
Set $\mu^{(k+1)}$ depending on $\alpha^{(k)}$;
\textbf{break}
}
\lElse{$\alpha^{(k)}\gets\frac{1}{2}\alpha^{(k)}$}
}
}
\textbf{return} $(\mathbf{x}^{(k)}, \mathbf{u}^{(k)})$
\end{algorithm}
\endgroup
\section{Convergence Analysis}
This section presents a global convergence result for FP-DDP.
Local convergence follows from the results in \cite{Baumgaertner2023a} and \cite{bergou2020}. 
Firstly, several auxiliary lemmas are introduced.
\begingroup\abovedisplayskip=0.0pt \belowdisplayskip=0.0pt
\begin{lemma}
\label{lem:characterization_optimality}
Let a trajectory $(\bar{\mathbf{x}},\bar{\mathbf{u}})$ be given that satisfies the dynamics \eqref{eq:focp_dynamics_constraint}. Then, the following statements hold:
\begin{enumerate}
    \item[(i)] If $(\bar{\mathbf{x}},\bar{\mathbf{u}})$ is not optimal for \eqref{eq:nlp_focp}, then the Hessian matrix constructed by Algorithm \ref{alg:ffp} is positive definite.
    \item[(ii)] The optimal solution of \eqref{eq:ocp_qp} is $(\mathbf{\delta x}, \mathbf{\delta u}) = 0$ if and only if $(\mathbf{\bar{x}}, \mathbf{\bar{u}})$ is a KKT point of \eqref{eq:nlp_focp}.
    \item[(iii)] The predicted reduction of \eqref{eq:ocp_qp} is zero, $m_f(\mathbf{\delta x}, \mathbf{\delta u}) = 0$, iff $(\mathbf{\delta x}, \mathbf{\delta u}) = 0$. Otherwise $m_f(\mathbf{\delta x}, \mathbf{\delta u}) > 0$.
\end{enumerate}
\end{lemma}
\endgroup
\begin{proof}
Since $(\bar{\mathbf{x}},\bar{\mathbf{u}})$ is not optimal, $f(\bar{\mathbf{x}},\bar{\mathbf{u}})$ is non-zero, i.e., the regularization term in \eqref{eq:regularization} is positive. Thus, referring to Sec.~\ref{sec:regularization} yields (i).
Result (ii) directly follows from comparing the KKT conditions of \eqref{eq:nlp_focp} and \eqref{eq:ocp_qp} with $(\mathbf{\delta x}, \mathbf{\delta u}) \!=\! 0$. 
For (iii): Since all iterates of Algorithm \ref{alg:ffp} are dynamically feasible, QP \eqref{eq:ocp_qp} has a feasible direction with $(\mathbf{\delta x}, \mathbf{\delta u}) = 0$ for which the objective function value of \eqref{eq:ocp_qp} is $0$.
The positive definiteness of the Hessian implies that there exists a unique solution to \eqref{eq:ocp_qp}.
Recalling that $m_f((\mathbf{\delta x}, \mathbf{\delta u}))=-f_{\mathrm{QP}}^*(\bar{\mathbf{x}},\bar{\mathbf{u}})$ shows that $m_f((\mathbf{\delta x}, \mathbf{\delta u})) = 0$ if $(\mathbf{\delta x}, \mathbf{\delta u}) = 0$ and otherwise $m_f((\mathbf{\delta x}, \mathbf{\delta u})) > 0$.
\end{proof}
\begingroup\abovedisplayskip=.5pt \belowdisplayskip=.5pt
\begin{lemma}
\label{lemma:ddp_estimate}
For the DDP step \eqref{eq:dms_control}, \eqref{eq:ddp_forward} it holds that
\begingroup\abovedisplayskip=1pt \belowdisplayskip=1pt
\begin{align*}
    (\hat{\mathbf{x}}(\alpha),\hat{\mathbf{u}}(\alpha)) - (\bar{\mathbf{x}},\bar{\mathbf{u}}) = \alpha (\mathbf{\delta x}, \mathbf{\delta u}) + \mathcal{O}(\alpha^2).
\end{align*}
\endgroup
\end{lemma}
\endgroup
\begin{proof}
    The lemma can be proved by induction and comparing the derivative of $(\hat{\mathbf{x}}(\alpha),\hat{\mathbf{u}}(\alpha))$ at $\alpha$ with $(\mathbf{\delta x}, \mathbf{\delta u})$ obtained from \eqref{eq:all_dms_equations}.
\end{proof}
\begin{lemma}[from \cite{Yakowitz1984}]
\label{lemma:objective_ddp_estimate}
For the objective function $F$ with eliminated states, it holds
\begingroup\abovedisplayskip=1pt \belowdisplayskip=1pt
\begin{align*}
F(\hat{\mathbf{u}}(\alpha)) - F(\bar{\mathbf{u}}) = \alpha f^*_{\mathrm{QP}}(\bar{\mathbf{x}},\bar{\mathbf{u}}) + \mathcal{O}(\alpha^2).
\end{align*}
\endgroup
\end{lemma}
From here, sufficient decrease and global convergence can be derived.
\begin{lemma}
\label{lemma:sufficiend_decrease}
Let a dynamically feasible trajectory $(\bar{\mathbf{x}},\bar{\mathbf{u}})$ be given that is not optimal and let $\eta\in(0,1)$.
Then there exists an $\alpha\in(0,1]$ such that the sufficient decrease condition \eqref{eq:sufficient_decrease} is satisfied.
\end{lemma}
\begin{proof}
Let the trial iterate be denoted by $(\hat{\mathbf{x}}(\alpha),\hat{\mathbf{u}}(\alpha))$. Noting that $F(\hat{\mathbf{u}}(\alpha)) = f(\hat{\mathbf{x}}(\alpha),\hat{\mathbf{u}}(\alpha))$, $F(\bar{\mathbf{u}}) = f(\bar{\mathbf{x}},\bar{\mathbf{u}})$, and $-f^*_{\mathrm{QP}}(\bar{\mathbf{x}},\bar{\mathbf{u}}) = m_f(\mathbf{\delta x}, \mathbf{\delta u})$ and using Lemma \ref{lemma:objective_ddp_estimate}, it follows 
$f(\bar{\mathbf{x}},\bar{\mathbf{u}}) - f(\hat{\mathbf{x}}(\alpha),\hat{\mathbf{u}}(\alpha)) \geq \alpha m_f(\mathbf{\delta x}, \mathbf{\delta u}) - \bar{\xi} \alpha^2$
for $\bar{\xi}>0$, i.e., $\alpha\leq((1-\eta)m_f(\mathbf{\delta x}, \mathbf{\delta u}))/\bar{\xi}$ yields the claim.
\end{proof}
\begin{theorem}[Global Convergence]
\label{theorem:global_convergence}
Suppose that $\nabla f$ is Lipschitz continuous and that $H^{(k)}$ is uniformly bounded for all $k$.
Then, Algorithm~\ref{alg:ffp} creates a sequence of iterates $\{(\mathbf{x}^{(k)}, \mathbf{u}^{(k)})\}_{k\in\mathbb{N}}$ and 
any accumulation point $(\mathbf{x}^*, \mathbf{u}^*)$ of the sequence is a stationary point for \eqref{eq:nlp_focp}.
\end{theorem}
\begin{proof}
From Lemma \ref{lemma:sufficiend_decrease} follows that Algorithm \ref{alg:ffp} is well-defined, i.e., a sequence $\{(\mathbf{x}^{(k)}, \mathbf{u}^{(k)})\}_{k\in\mathbb{N}}$ is found and that the sufficient decrease condition \eqref{eq:sufficient_decrease} is satisfied in every iteration $k$.
Summing for $k=0$ to $N$ and using a telescopic sum argument yields
\begingroup\abovedisplayskip=3pt \belowdisplayskip=3pt
\begin{align*}
    \resizebox{.9\hsize}{!}{$\eta \sum_{k=0}^N \alpha^{(k)} m_f^{(k)}(\mathbf{\delta x}^{(k)}, \mathbf{\delta u}^{(k)})\leq f^{(0)} - f^{(N)} \leq f^{(0)}$}
\end{align*}
\endgroup
since the lower bound of $f$ is $0$. Taking the limit yields
\begingroup\abovedisplayskip=3pt \belowdisplayskip=3pt
\begin{equation}
\label{eq:reihe_bounded}
\resizebox{.6\hsize}{!}{$\sum_{k=0}^{\infty} \alpha^{(k)} m_f^{(k)}(\mathbf{\delta x}^{(k)}, \mathbf{\delta u}^{(k)}) < \infty,$}
\end{equation}
\endgroup
implying $\alpha^{(k)} m_f^{(k)}(\mathbf{\delta x}^{(k)}, \mathbf{\delta u}^{(k)})\to0$ since Lemma \ref{lem:characterization_optimality} (iii) guarantees non-negativity of the summands. To finish the proof it needs to be shown that $\alpha^{(k)}$ is bounded away from zero for $k\to\infty$.
For a contradiction, assume that the predicted reduction $m_f^{(k)}$ does not converge to 0, i.e., $(\mathbf{\delta x}^{(k)}, \mathbf{\delta u}^{(k)})\not\to0$. This implies $\alpha^{(k)}\to0$.
If $\alpha^{(k)}=1$, there is nothing to show.
So, assume that for $\alpha^{(k)}<1$ the Armijo condition \eqref{eq:sufficient_decrease} is satisfied and dropping the index $k$, then for $2\alpha$, \eqref{eq:sufficient_decrease} is not satisfied. Equation \eqref{eq:sufficient_decrease} being not satisfied is equivalent to
$F(\mathbf{\hat{u}}(2\alpha)) - F(\mathbf{\bar{u}}) > 2\eta\alpha\nabla F(\mathbf{\bar{u}})^{\top}\delta \mathbf{u}$.
From here, knowing that $F$ is once continuously differentiable with Lipschitz gradients and bounded Hessian a standard unconstrained optimization proof, e.g., \cite[Thm~9.1]{Diehl2016}, can be applied to bound $\alpha$ from below.
The DDP step needs to be expressed in terms of $\delta \mathbf{u}$ using Lemma~\ref{lemma:ddp_estimate}. Then, it can be shown that $\alpha > M \Vert \delta\mathbf{u}\Vert^2$ for $M>0$. From this, it follows that $\alpha\not\to 0$ which is a contradiction to \eqref{eq:reihe_bounded}. As a consequence, $m_f^{(k)}(\mathbf{\delta x}^{(k)}, \mathbf{\delta u}^{(k)})\to0$ and convergence follows from Lemma \ref{lem:characterization_optimality}.
\end{proof}
Theorem \ref{theorem:global_convergence} guarantees convergence of FP-DDP to a KKT point, but in general, it does not guarantee that $f(\mathbf{x},\mathbf{u})=0$ at the solution.
If the OCP \eqref{eq:nlp_ocp} is convex, FOCP \eqref{eq:nlp_focp} is also convex.
If therefore, the original OCP has feasible points, all local solutions of \eqref{eq:nlp_focp} are global solutions and fulfill $f(\mathbf{x},\mathbf{u})=0$. In this case, FP-DDP is guaranteed to converge to a feasible point.
For nonconvex problems, this cannot be guaranteed since many local solutions can exist. If the optimal objective function value is nonzero, this indicates that the problem might be locally infeasible \cite{Waechter2006}. In the simulation examples presented next, this case never occurred. FP-DDP always found a feasible solution.
\section{Simulation Results}
\label{sec:simulation_results}
This section demonstrates the capabilities of FP-DDP on two test problems, (i) a fixed-time P2P motion of an unstable system, (ii) a free-time P2P motion of a cart pendulum system with a circular obstacle.
\subsection{Implementation}
The OCPs are formulated using the rapid prototyping toolbox \texttt{rockit} \cite{Gillis2020} built on top of \texttt{CasADi} \cite{Andersson2019}.
The FP-DDP algorithm is implemented within \texttt{acados}~\cite{Verschueren2022} utilizing the Riccati recursion from \texttt{HPIPM}~\cite{Frison2020a}.
The code to reproduce the results and a prototypical python implementation of FP-DDP are publicly available~\cite{fp_ddp_code}.
The parameters of FP-DDP were chosen as $\eta=10^{-6}$, $\alpha_{\min}=10^{-17}$, $\mu_{\min}=10^{-16}$, $\mu^{(0)}=10^{-3}$, $\lambda=5$, $\varepsilon_{\mathrm{F}}=10^{-12}$, and $\varepsilon_{\mathrm{S}}=10^{-8}$.
In example (ii), the \texttt{IPOPT} default settings were used. For the truncated Newton method of \texttt{scipy.optimize.minimize}, $\mathrm{tol}=10^{-8}$ was used.
All simulations were carried out using an i7-10810U CPU.
\subsection{Fixed-Time P2P Motion of Unstable System}
In this test problem, FP-DDP is compared against a Direct Single Shooting (DSS) implementation using the globalization of FP-DDP and a Direct Multiple Shooting (DMS) implementation using the steering penalty method described in \cite{Byrd2009}. 
The example originates from \cite{chen1998} and is adopted from \cite{Baumgaertner2023a}. 
The dynamics are given by
\begingroup\abovedisplayskip=2pt \belowdisplayskip=2pt
\begin{align*}
\dot{x}_1\!=\!x_2\!+\!u\left(\zeta+(1\shortminus\zeta) x_2\right),\:\dot{x}_2\!=\!x_1\!+\!u\left(\zeta\shortminus4(1\shortminus\zeta) x_2\right),
\end{align*}
\endgroup
with $\zeta=0.7$.
The problem is discretized using an explicit fourth-order Runge-Kutta integrator with $10$ internal integration steps, a time interval size of $h=0.25$, and an OCP horizon $N=20$.
Control bounds $\vert u_i\vert\leq u_{\max}=1.5$ are imposed.
The task is to perform a point-to-point motion of the dynamical system from the initial state $\bar{x}_0=(0.42,0.45)^{\top}$ to the terminal state $\bar{x}_N=(0.0, 0.1)^{\top}$.
As in \cite{Baumgaertner2023a}, a dynamically feasible initial guess is obtained by a forward simulation of the system using the feedback law which is retrieved from an LQR controller linearized at the equilibrium point $\mathbf{x}=0$ and $\mathbf{u}=0$.
All approaches can find a feasible trajectory.
\begingroup\abovedisplayskip=3pt \belowdisplayskip=3pt
\begin{figure}[htpb]
\centering
\includegraphics[width=0.98\columnwidth]{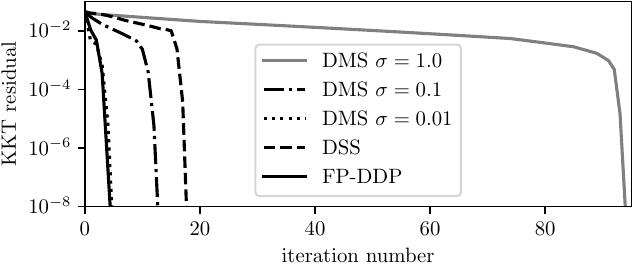}
\caption{KKT residuals for DMS with penalty parameter $\sigma\in\{1.0, 0.1, 0.01\}$, DSS, and FP-DDP.}
\label{fig:chen_allgoewer_kkt_residual}
\end{figure}
\endgroup
Fig. \ref{fig:chen_allgoewer_kkt_residual} shows the convergence of the different methods. FP-DDP takes 5 iterations to converge to a feasible solution. 
Interestingly, always full steps are taken.
DSS takes 18 iterations until the optimal solution is found. 
Since the underlying system is highly unstable, the open-loop forward simulation in DSS yields trajectories that are far from the solution. 
This results in smaller steps for DSS and therefore slower convergence. 
Looking at DMS reveals that the choice of the penalty parameter highly influences the convergence speed. 
Setting the initial penalty parameter $\sigma=1$ as in \cite{Byrd2009} results in convergence in 94 iterations. The Maratos effect seems to occur.
For $\sigma=0.01$, the algorithm convergences within 5 iterations. 
This demonstrates the dependence of merit-function-based algorithms on their penalty parameter whereas FP-DDP does not require this parameter choice.

\subsection{Free-Time P2P Motion of Cart Pendulum}
This section discusses a free-time cart pendulum problem with obstacle avoidance.
The system dynamics and parameters are given in \cite{Verschueren2016}.
The states $x=(p, v, \theta, \omega)^{\top}\in\mathbb{R}^4$ consist of position, velocity, pole angle, and pole angular velocity.
The control (force) input is given by $u=F\in\mathbb{R}$.
The free-time $T$ is modeled as an additional state using time scaling.
The task is to move from an initial upward pole position corresponding to $\bar{x}_1=(0, 0, 0, 0)^{\top}$ to the final upward position $\bar{x}_{N+1}=(5.0, 0, 0, 0)^{\top}$.
In addition, some bound constraints $\underline{x}\leq x_i\leq\bar{x}$ with $i=2,\ldots, N$ and $\underline{x}=(-1.0, -2.0, -\frac{\pi}{4}, -0.5)$, $\bar{x}=(6.0, 2.0, \frac{\pi}{4}, 0.5)$, as well as $|u_i|\leq5.0$ for $i=1,\ldots, N$ are imposed.
During the motion, the endpoint of the pole has to avoid a circular obstacle with its center at $(\tilde{p}, 0.9)^{\top}$, $\tilde{p}\in\mathbb{R}$ and with a radius of $0.3$. 
The horizontal position of the center of the obstacle is chosen from $100$ linearly spaced values between $0.7$ and $4.3$, i.e., in total, $100$ problems are solved.

FP-DDP is compared against \texttt{IPOPT} \cite{Waechter2006}, which solves the same problem formulation as FP-DPP, and against \texttt{scipy.optimize.minimize} using a Truncated-Newton method \cite{Nocedal2006}, \cite{scipy2020} and solving the unconstrained NLP of \texttt{SEQUOIA} which is given in Table~\ref{tab:comparison_feasibility_problems}.
\texttt{IPOPT} is called through its \texttt{CasADi} interface which provides the derivative information and enables code generation. Note that code generation increases execution speed, but does not alter the algorithm. 
The derivatives for the unconstrained NLP as well as the GGN Hessian are manually built in CasADi and passed to \texttt{scipy}.
\texttt{IPOPT} uses exact Hessian information. 
The solvers were initialized with a constant trajectory at the initial position and constant value $T=5.0$.

To compare the different algorithms, the Dolan and Mor\'e performance profiles \cite{Dolan2001} are used. 
A problem is solved correctly if the final iterate is dynamically feasible and has an objective function value smaller than $10^{-8}$. 
Otherwise, it is a failure. 
As performance metrics, we regard the wall time and the number of Hessian evaluations. Fig.~\ref{fig:performance_plot_cart_pendulum} shows the results of the different solvers. 
The vertical axis indicates the number of test problems that were solved. 
The horizontal axis on the left shows the absolute wall time and on the right side the ratio of Hessian evaluations needed to converge relative to the best solver, i.e., the solver did not need more than its Hessian evaluation ratio times the number of Hessian evaluations of the best solver until convergence. 
Higher values on the vertical axis in plot (b) indicate better performance of an algorithm relative to another algorithm.
\begingroup\abovedisplayskip=3pt \belowdisplayskip=2pt
\begin{figure}[htpb]
\centering
\includegraphics[width=\columnwidth]{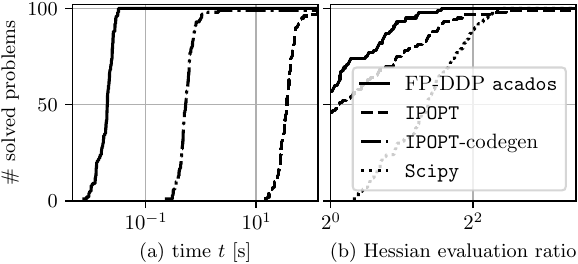}
\setlength{\belowcaptionskip}{-8pt}
\caption{Performance plot of wall time (left) and Hessian evaluations ratio with respect to the best solver (right).}
\label{fig:performance_plot_cart_pendulum}
\end{figure}
\endgroup

It can be seen that FP-DDP outperforms the other algorithms and is in almost 55\% of the cases the algorithm with the least amount of Hessian evaluations until convergence is achieved.
\texttt{IPOPT} fails once in the restoration phase.
Apart from that all solvers are robust and converge towards feasible trajectories. 
Overall, the wall time of FP-DDP implemented within \texttt{acados} is an order of magnitude lower than the code generated \texttt{IPOPT}. 
Due to the slow run time, \texttt{scipy} was excluded from the wall time plot. 
This demonstrates the strong performance of FP-DDP.
\section{Conclusion}
\label{sec:conclusion}
This paper introduced an efficient algorithm for finding feasible trajectories in direct optimal control based on DDP and an Armijo-type line search. Theoretical convergence guarantees were derived and the efficiency of the algorithm was demonstrated against state-of-the-art methods. Possible future research areas include the embedding of FP-DDP into an OCP NLP solver.
\bibliography{bibliography}

\begin{thebibliography}{10}

\bibitem{Andersson2019}
J.~Andersson, J.~Gillis, G.~Horn, J.~Rawlings, and M.~Diehl.
\newblock {CasADi} -- a software framework for nonlinear optimization and
  optimal control.
\newblock {\em Math. Program. Comput.}, 11(1):1--36, 2019.

\bibitem{Baumgaertner2023a}
K.~Baumg{\"a}rtner, F.~Messerer, and M.~Diehl.
\newblock A unified local convergence analysis of differential dynamic
  programming, direct single shooting, and direct multiple shooting.
\newblock In {\em Proc. Eur. Control Conf. (ECC)}, 2023.

\bibitem{bergou2020}
E.~Bergou, Y.~Diouane, and V.~Kungurtsev.
\newblock Convergence and complexity analysis of a levenberg–marquardt
  algorithm for inverse problems.
\newblock {\em J. Optim. Theory Appl.}, 185, 06 2020.

\bibitem{Byrd2009}
R.~Byrd, G.~Lopez-Calva, and J.~Nocedal.
\newblock A line search exact penalty method using steering rules.
\newblock {\em Math. Program.}, 133:1--35, 01 2009.

\bibitem{chen1998}
H.~Chen and F.~Allgo\"wer.
\newblock A quasi-infinite horizon nonlinear model predictive control scheme
  with guaranteed stability.
\newblock {\em Automatica}, 34(10):1205--1217, 1998.

\bibitem{Diehl2016}
M.~Diehl.
\newblock {\em Lecture Notes on Numerical Optimization}.
\newblock 2016.

\bibitem{Dolan2001}
E.~Dolan and J.~Moré.
\newblock Benchmarking optimization software with performance profiles.
\newblock {\em Math. Program.}, 91, 03 2001.

\bibitem{fletcher1997}
R.~Fletcher and S.~Leyffer.
\newblock Nonlinear programming without a penalty function.
\newblock {\em Math. Program.}, 91:239--269, 02 1999.

\bibitem{Frison2020a}
G.~Frison and M.~Diehl.
\newblock {HPIPM}: a high-performance quadratic programming framework for model
  predictive control.
\newblock In {\em Proc. IFAC World Congr.}, Berlin, Germany, July 2020.

\bibitem{Gillis2020}
J.~Gillis, B.~Vandewal, G.~Pipeleers, and J.~Swevers.
\newblock Effortless modeling of optimal control problems with rockit.
\newblock In {\em Proc. 39th Benelux Meet. Sys. Control}, Elspeet, the
  Netherlands, July 2020.

\bibitem{Hintermueller2010}
M.~Hinterm{\"u}ller.
\newblock Semismooth newton methods and applications.
\newblock {\em Department of Mathematics, Humboldt-University of Berlin}, 2010.

\bibitem{fp_ddp_code}
D.~Kiessling.
\newblock {FP}-{DDP} code.
\newblock \url{github.com/david0oo/fp_ddp_python}.

\bibitem{Mastalli2019}
C.~Mastalli et~al.
\newblock Crocoddyl: An efficient and versatile framework for multi-contact
  optimal control.
\newblock In {\em Proc. IEEE Int. Conf. Robot.Autom. (ICRA)}, 2020.

\bibitem{mayne1966}
D.~Mayne.
\newblock A second-order gradient method for determining optimal trajectories
  of non-linear discrete-time systems.
\newblock {\em Int. J. Control}, 3(1):85--95, 1966.

\bibitem{murray1984}
D.~Murray and S.~Yakowitz.
\newblock Differential dynamic programming and newton's method for discrete
  optimal control problems.
\newblock {\em J. Optim. Theory Appl.}, 43:395--414, 1984.

\bibitem{Nita2023}
L.~Nita and E.~Carrigan.
\newblock {SEQUOIA}: A sequential algorithm providing feasibility guarantees
  for constrained optimization.
\newblock In {\em Proc. IFAC World Congr.}, Yokohama, Japan, 2023.

\bibitem{Nocedal2006}
J.~Nocedal and S.~J. Wright.
\newblock {\em Numerical Optimization}.
\newblock Springer Ser. Oper. Res. Financ. Eng. Springer, 2 edition, 2006.

\bibitem{Tenny2004}
M.~Tenny, S.~Wright, and J.~Rawlings.
\newblock {N}onlinear model predictive control via feasibility-perturbed
  sequential quadratic programming.
\newblock {\em Comput. Optim. Appl.}, 28:87--121, 2004.

\bibitem{Verschueren2022}
R.~Verschueren et~al.
\newblock acados -- a modular open-source framework for fast embedded optimal
  control.
\newblock {\em Math. Prog. Comput.}, 14:147--183, 2021.

\bibitem{Verschueren2016}
R.~Verschueren, N.~van Duijkeren, R.~Quirynen, and M.~Diehl.
\newblock Exploiting convexity in direct optimal control: a sequential convex
  quadratic programming method.
\newblock In {\em Proc. IEEE Conf. Decis. Control (CDC)}, 2016.

\bibitem{scipy2020}
P.~Virtanen et~al.
\newblock {{SciPy} 1.0: Fundamental Algorithms for Scientific Computing in
  Python}.
\newblock {\em Nat. Methods}, 17:261--272, 2020.

\bibitem{Waechter2006}
A.~W\"achter and L.~Biegler.
\newblock On the implementation of an interior-point filter line-search
  algorithm for large-scale nonlinear programming.
\newblock {\em Math. Program.}, 106(1):25--57, 2006.

\bibitem{Wright2004}
S.~Wright and M.~Tenny.
\newblock {A} feasible trust-region sequential quadratic programming algorithm.
\newblock {\em SIAM J. Optim.}, 14:1074--1105, 1 2004.

\bibitem{Yakowitz1984}
S.~Yakowitz and B.~Rutherford.
\newblock Computational aspects of discrete-time optimal control.
\newblock {\em Appl. Math. Comp.}, 15(1):29--45, 1984.

\bibitem{Tits1992}
J.~Zhou and A.~Tits.
\newblock User's guide for {FSQP} version 3.0c: A {FORTRAN} code for solving
  constrained nonlinear (minimax) optimization problems, generating iterates
  satisfying all inequality and linear constraints.
\newblock 1992.

\end{thebibliography}
\bibliographystyle{plain}
\addtolength{\textheight}{-12cm}
\end{document}